\documentclass[a4paper,12pt]{amsart}

\usepackage[all]{xy}
\usepackage{graphicx}

\usepackage{t1enc}

\def\Cal{\mathcal}
\let\La\Lambda
\let\o\circ
\newcommand{\im}{\operatorname{im}}
\newcommand{\Rho}{P}

\newcommand{\ce}{{\Cal E}}

\newcommand{\cT}{{\mathcal T}}

\newcommand{\rpl}                         
{\mbox{$
\begin{picture}(12.7,8)(-.5,-1)
\put(0,0.2){$+$}
\put(4.2,2.8){\oval(8,8)[r]}
\end{picture}$}}

\def\cV{\mathcal{V}}

\def\al{\alpha}

\newcommand{\si}{\sigma}

\newcommand{\End}{\operatorname{End}}

\newcommand{\lpl}{
  \mbox{$
  \begin{picture}(12.7,8)(-.5,-1)
  \put(2,0.2){$+$}
  \put(6.2,2.8){\oval(8,8)[l]}
  \end{picture}$}}

\newtheorem{theorem}{Theorem}[section]
\newtheorem{lemma}[theorem]{Lemma}
\newtheorem{proposition}[theorem]{Proposition}
\newtheorem{corollary}[theorem]{Corollary}

\theoremstyle{definition}
\newtheorem{definition}[theorem]{Definition}

\theoremstyle{remark}
\newtheorem{remark}[theorem]{\rm\bf Remark}
\newtheorem*{definition*}{\rm\bf Definition}

\newcommand{\nn}[1]{(\ref{#1})}

\newcommand{\vol}{\large\boldsymbol{\epsilon}}

\def\sideremark#1{\ifvmode\leavevmode\fi\vadjust{\vbox to0pt{\vss
 \hbox to 0pt{\hskip\hsize\hskip1em
 \vbox{\hsize3cm\tiny\raggedright\pretolerance10000
  \noindent #1\hfill}\hss}\vbox to8pt{\vfil}\vss}}}%

\author{A.\ \v Cap, A.\ R. Gover \& H.\ R.\ Macbeth}
\title{Einstein metrics in projective geometry}
\begin{document}

\address{A.\v C.: Faculty of Mathematics\\
University of Vienna\\
Nordbergstr. 15\\
1090 Wien\\
Austria\\
A.R.G.: Department of Mathematics\\
  The University of Auckland\\
  Private Bag 92019\\
  Auckland 1142\\
  New Zealand;\\
Mathematical Sciences Institute\\
Australian National University \\ ACT 0200, Australia\\
H.R.M.: Department of Mathematics\\
Princeton University\\
Princeton, NJ 08544\\
USA}
\email{Andreas.Cap@univie.ac.at}
\email{r.gover@auckland.ac.nz}
\email{macbeth@math.princeton.edu}


\begin{abstract}
  It is well known that pseudo--Riemannian metrics in the projective
  class of a given torsion free affine connection can be obtained from
  (and are equivalent to) the solutions of a certain overdetermined
  projectively invariant differential equation. This equation is a
  special case of a so--called first BGG equation. The general theory
  of such equations singles out a subclass of so--called normal
  solutions. We prove that non-degenerate normal solutions are
  equivalent to pseudo--Riemannian Einstein metrics in the projective
  class and observe that this connects to natural projective extensions
  of the Einstein condition.
\end{abstract}

\maketitle \pagestyle{myheadings} \markboth{\v Cap, Gover \& Macbeth}{Normality}

\thanks{A\v C \& ARG gratefully acknowledge support from the Royal
  Society of New Zealand via Marsden Grant 10-UOA-113; A\v C
  gratefully acknowledges support by project P23244-N13 of the ``Fonds
  zur F\"orderung der wissenschaftlichen For\-schung'' (FWF); A\v C and
  HRM are grateful for the hospitality of the University of Auckland. }

\section{Introduction}\label{intro}

Suppose that $\nabla$ is a torsion-free connection on a manifold
$M^n$, $n\geq 2$ , and consider its geodesics as unparametrised
curves. The problem of whether these agree with the (unparametrised)
geodesics of a pseudo-Riemannian metric is the classical problem of
metrizability of projective structures which has attracted recent
interest \cite{BDE,EM,Liouville,Mikes,NurMet,Sinjukov}.

Recall that torsion--free connections $\nabla$ and $\widehat{\nabla}$
are said to be projectively equivalent if they have the same geodesics
as unparameterised curves.  A projective structure on a manifold $M$
(of dimension $n\geq 2$) is a projective equivalence class $p$ of
connections.

As is usual in projective geometry, we write $\ce(1)$ for a choice of
line bundle with $(-2n-2)$nd power the square of the canonical bundle
$\Lambda^n T^*M$.  Observe that any connection $\nabla \in p$
determines a connection on $\ce(1)$ as well as its real powers
$\ce(w)$, $w\in \mathbb{R}$; we call $\ce(w)$ the bundle of projective
densities of weight $w$. Given any bundle $\mathcal{B}$ we shall write
$\mathcal{B}(w)$ as a shorthand notation for $\mathcal{B}\otimes
\ce(w)$.

For simplicity here we suppose that $M$ is connected and orientable.
We say that a connection $\nabla$ is {\em special} if it preserves a
volume form $\epsilon$ on $M$.
On the other hand, suppose that
$\epsilon$ is a volume form on $M$ and $\tilde\nabla$ is any
connection on $TM$. Considering the induced connection on $\La^nT^*M$,
we can write $\nabla\epsilon$ as $\al\epsilon$ for some one--form
$\al\in\Omega^1(M)$. Then one easily verifies that 
$$
\nabla_\xi\eta:=\tilde\nabla_\xi\eta
+\tfrac1{n+1}\al(\xi)\eta+\tfrac1{n+1}\al(\eta)\xi 
$$ is a connection in the projective class of $\tilde\nabla$ for which
$\epsilon$ is parallel. 
Henceforth we use $p$ to denote the
equivalence class of projectively related special connections.  For
convenience we shall often use the Penrose abstract index notation and
write $\ce^{(bc)}$ for the symmetric tensor power of the tangent
bundle (otherwise written $ S^2 (TM)$) and $(\ce_a^{(bc)})_0$ for the
trace-free part of $T^*M\otimes S^2 (TM)$.

Consider the differential operator
$$
D_a: \ce^{(bc)}(-2) \to (\ce_a{}^{(bc)})_0(-2), \quad \mbox{given by
}\quad \si^{bc}\mapsto \operatorname{trace-free}\left( \nabla_a
  \si^{bc}\right).
$$ It is an easy exercise to verify that $D$ is a projectively
invariant differential operator in that it is independent of the
choice $\nabla\in p$.  Part of the importance of $D$ derives from the
following result due to Mikes and Sinjukov \cite{Mikes,Sinjukov}.
\begin{theorem}\label{EM1}
Suppose that $n\geq 2$ and $\nabla$ is a special torsion-free
connection on $M$. Then $\nabla$ is projectively equivalent to a
Levi-Civita connection if and only if there is a non-degenerate
solution $\si$ to the equation
\begin{equation}\label{metr}
D\si=0.
\end{equation}
\end{theorem}
\noindent Here $\si$ {\em non-degenerate} means that it is
non-degenerate as a bilinear form on $T^*M (1)$. Our presentation of
the Theorem here follows the treatment \cite{EM} of
Eastwood-Matveev. 

Let us write $\vol_{a_1a_2\cdots a_n}$ for the canonical section of
$\Lambda^nT^*M(n+1)$ which gives the tautological bundle map
$\Lambda^nTM\to \ce(n+1)$.  Observe that each section $\si^{ab}$ in
$\ce^{(ab)}(-2)$ canonically determines a section $\tau^\si\in
\ce(2)$, by taking its determinant using $\vol$:
\begin{equation} \label{taudef}
\si^{ab}\mapsto \tau^\si:=\si^{a_1b_1}\cdots \si^{a_nb_n}\vol_{a_1\cdots a_n}\vol_{b_1\cdots b_n}.
\end{equation}
We may form
\begin{equation} \label{invm}
\tau^\si \si^{ab}
\end{equation}
and in the case that $\si^{ab}$ is non-degenerate taking the inverse
of this yields a metric that we shall denote $g^\si_{ab}$. This
construction is clearly invertible and a metric $g_{ab}$ determines a
non-degenerate section $\si^{ab}\in \ce^{(ab)}(-2)$.  We are
interested in the metric $g^\si$ when $\si $ is a solution to
\nn{metr}. Indeed, the Levi-Civita connection mentioned in the Theorem
is the Levi-Civita connection for $g^\si$.

Now the projectively invariant differential operator $D$ arises from a
very general construction, namely as the first operator in a
Bernstein--Gelfand--Gelfand (BGG) sequence. For the definition and
general construction of these sequences see \cite{CSS,CD}.  For any
first BGG equation there is a special class of solutions known as
normal solutions, see \cite{Leitner}. These have striking properties,
see \cite{CGHjlms,CGH,CGHpoly,Gal}, but in general it is unclear how
restrictive the normality condition is, and in particular how commonly
normal solutions are available. The aim of this article is to analyse
the normality condition for solutions of \nn{metr}. This needs only
very basic facts on BGG sequences and simple elementary
considerations, and furthermore the answer is significant and
appealing. To explain these terms and prepare for that discussion we
need some elements of tractor calculus, an invariant calculus for
projective structures.

\subsection{Acknowledgements} We are grateful for the referee's 
insightful comments. 

\section{projective tractor calculus}
Consider the first jet prolongation
$J^1\ce(1)\to M$ of the density bundle. By definition, its fiber over
$x\in M$ consists of all one--jets $j^1_x\sigma$ of local smooth
sections $\sigma\in\Gamma(\ce(1))$ defined in a neighborhood of
$x$. Here for two sections $\sigma$ and $\tilde\sigma$ we have
$j^1_x\sigma=j^1_x\tilde\sigma$ if and only if in one --- or
equivalently any --- local chart the sections $\sigma$ and
$\tilde\sigma$ have the same Taylor--development in $x$ up to first
order. In particular, mapping $j^1_x\sigma$ to $\sigma(x)$ defines a
surjective bundle map $J^1\ce(1)\to\ce(1)$, called the \textit{jet
  projection}. If $j^1_x\sigma$ lies in the kernel of this projection,
so $\sigma(x)=0$ then the value $\nabla\sigma(x)\in
T^*_xM\otimes\ce_x(1)$ is the same for all linear connections $\nabla$
on the vector bundle $\ce(1)$. This identifies the kernel of the jet
projection with the bundle $T^*M\otimes\ce(1)$. (See for example
\cite{palais} for a general development of jet bundles.)

In an abstract index notation let us write $\ce_A$ for $J^1\ce(1)$ and
$\ce^A$ for the dual vector bundle. Then we can view the jet
projection as a canonical section $X^A$ of the bundle
$\ce^A\otimes\ce(1)=\ce^A(1)$. Likewise, the inclusion of the kernel
of this projection can be viewed as a canonical bundle map
$\ce_a(1)\to\ce_A$, which we denote by $Z_A{}^a$. Thus the jet exact
sequence (at 1-jets) is written in this context as
\begin{equation}\label{euler}
0\to \ce_a(1)\stackrel{Z_A{}^a}{\to} \ce_A \stackrel{X^A}{\to}\ce(1)\to 0.
\end{equation}
We write $\ce_A=\ce(1)\lpl \ce_a(1)$ to summarise the composition
structure in \nn{euler}.  As mentioned, any connection $\nabla \in p$
determines a connection on $\ce(1)$, and this is precisely a splitting
of \nn{euler}.  Thus given such a choice we have the direct sum
decomposition $\ce_A \stackrel{\nabla}{=} \ce(1)\oplus \ce_a(1) $ with
respect to which we define a connection by
\begin{equation}\label{pconn}
\nabla^{\mathcal{T}^*}_a \binom{\si}{\mu_b}
:= \binom{ \nabla_a \si -\mu_a}{\nabla_a \mu_b + P_{ab} \si}.
\end{equation}
Here $P_{ab}$ is the projective Schouten tensor and, with
$R_{ab}{}^c{}_d$ denoting the curvature of $\nabla$, is related to the
Ricci tensor $R_{ab}:=R_{ca}{}^c{}_b$ by
$(n-1)P_{ab}=R_{ab}$.  It turns out that \nn{pconn} is
independent of the choice $\nabla \in p$, and so
$\nabla^{\mathcal{T}^*}$ is determined canonically by the projective
structure $p$. We have followed the construction of \cite{BEG}, but
this {\em cotractor connection} is due to  \cite{Thomas}. It  is
equivalent to the normal Cartan connection (of \cite{Cartan}) for the
Cartan structure of type $(G,P)$, see \cite{CapGoTAMS}. Thus we shall
also term $\ce_A$ the {\em cotractor bundle}, and we note the dual
{\em tractor bundle} $\ce^A$ (or in index free notation $\mathcal{T}$)
has canonically the dual {\em tractor connection}: in terms of a
splitting dual to that above this is given by
\begin{equation}\label{tconn}
\nabla^\cT_a \left( \begin{array}{c} \nu^b\\
\rho
\end{array}\right) =
\left( \begin{array}{c} \nabla_a\nu^b + \rho \delta^b_a\\
\nabla_a \rho - P_{ab}\nu^b
\end{array}\right).
\end{equation}

Now consider $\ce^{( BC)}=S^2\mathcal{T}$. It follows immediately that
this has the composition series
$$
\ce^{(bc)}(-2)\lpl\ce^b(-2)\lpl \ce(-2),
$$
and the normal tractor connection is given on $S^2\mathcal{T}$ by
\begin{equation}\label{s2conn}
\nabla_a^{\cT}\left( \begin{array}{c}
\si^{bc}\\
\mu^b\\
\rho
\end{array}
\right)=
\left( \begin{array}{c}
\nabla_a \si^{bc} + \delta^b_a \mu^c + \delta^c_a \mu^b \\
\nabla_a \mu^b + \delta^b_a\rho - P_{ac}\si^{bc}  \\
\nabla_a \rho - 2 P_{ab}\mu^b
\end{array}
\right).
\end{equation}

\subsection{The Kostant codifferential}
The tractor connection on $S^2\Cal T$, and more generally on a tractor
bundle $\Cal V$, which is formed by tensorial constructions from
$\ce^A$ and $\ce_A$, extends to the covariant exterior derivative on
$\Cal V$--valued forms. Thus one so obtains a twisting of the de Rham
sequence by $\Cal V$, and this is central in the usual construction of
BGG sequences. At the next stage of the construction, a second ingredient is
needed, as follows.

Note that from \nn{euler} it follows that there is a
canonical (projectively invariant) map
\begin{equation}\label{Xf}
\mathbb{X}: T^*M\to \End (\cT) \quad \mbox{given by} \quad u_b
\mapsto X^AZ_B{}^b u_b .
\end{equation}
Since sections of $\End (\cT)$ act on any tractor bundle in the
obvious (tensorial) way, we thus obtain via $\mathbb{X}$ a canonical
action of $T^*M$ on any tractor bundle $\cV$. This induces a sequence
of natural bundle maps
$$
\partial^*:\La^kT^*M\otimes\Cal V\to\La^{k-1}T^*M\otimes\Cal V, \quad k=1,\cdots , n,
$$ on $\cV$-valued differential form bundles, but going in the
opposite direction to the twisted de Rham sequence. This a special
case of a Kostant codifferential and satisfies
$\partial^*\o\partial^*=0$, so it leads to natural subquotient
bundles $H_k(M,\Cal V):=\ker(\partial^*)/\im(\partial^*)$. (The
notation for these bundles is due to the fact that they are induced by
certain Lie algebra homology groups, but this is not relevant for our
purposes.)

In the case of $\Cal V= S^2\Cal T$, which is relevant for our purposes, the end
of this sequence has
the form
$$
\begin{pmatrix} \ce^{(ab)}(-2) \\ \ce^a(-2) \\ \ce(-2)\end{pmatrix}
\overset{\partial^*}\longleftarrow
\begin{pmatrix} \ce^{(bc)}_a(-2) \\ \ce_a^b(-2) \\
  \ce_a(-2)\end{pmatrix}
\overset{\partial^*}\longleftarrow
\begin{pmatrix} \ce^{(cd)}_{[ab]}(-2) \\ \ce^c_{[ab]}(-2) \\
  \ce_{[ab]}(-2)\end{pmatrix} ,
$$ where we have used a vector notation analogous to \eqref{s2conn}.
From the general theory (or indeed the formula \nn{Xf}) it follows
that $\partial^*$ maps each row in some column to the row below in the
next column to the left, and that all the bundle maps are natural. All
we need to now here is the following result.
\begin{lemma}\label{H1lem}
In terms of the composition series $\ce_a^{(bc)}(-2)\lpl\ce_a^b(-2)\lpl
\ce_a(-2)$ for $T^*M\otimes S^2\Cal T$ we have
$$
\im(\partial^*)=(\ce_a^b)_0(-2)\lpl\ce_a(-2)\subset
(\ce_a^{(bc)})_0(-2)\lpl(\ce_a^b)_0(-2)\lpl
\ce_a(-2)=\ker(\partial^*)
$$
\end{lemma}
\begin{proof}
  From what we know about $\partial^*$, we see that
  $\partial^*:T^*M\otimes S^2\Cal T\to S^2\Cal T$ can only induce some
  multiples of the trace maps $\ce_a^b(-2)\to\ce(-2)$ and
  $\ce_a^{(bc)}(-2)\to\ce^c(-2)$ on the two upper slots. Likewise, in
  the next step, there can only be multiples of the canonical trace
  maps applied to the two upper slots.

  Now there are some simple general facts about the homology of the
  $\partial^*$--sequence, see e.g.~\cite{BCEG}. These imply that the
  homology in degree zero coincides with the irreducible quotient
  bundle $\ce^{(ab)}(-2)$ and the homology in degree one is
  $(\ce_a^{(bc)})_0(-2)$. This implies that all the bundle maps from
  above are actually non--zero multiples of the trace maps, and hence
  the claim.
\end{proof}

\section{BGG sequences and normal solutions}\label{BGG}

Let us write $\Pi:\ce^{(BC)}\to \ce^{(bc)}(-2)$ for the canonical
projectively invariant map onto the quotient; explicitly this is given
by $H^{BC}\mapsto Z_B{}^bZ_C{}^cH^{BC}$. The key step to the
construction of BGG sequences is the construction of a differential
splitting to the tensorial operator on sections induced by this
projection. Phrased for the case of $S^2\Cal T$, this reads as

\begin{proposition}\label{nprop}
  For a smooth section $\si$ of $\ce^{(bc)}(-2)$ there is a unique
  smooth section $L(\si)$ of $\ce^{(AB)}$ such that $\Pi(L(\si))=\si$
  and $\partial^*(\nabla L(\si))=0$. This defines a projectively
  invariant differential operator $L$, and $D(\si)$ is given by
  projecting $\nabla (L(\si))$ to the quotient bundle
  $\ker(\partial^*)/\im(\partial)^*\cong (\ce_a^{(bc)})_0(-2)$.
\end{proposition}

In the special case needed here, this can also be proved by a direct
computation. Indeed, given $\si=\si^{ab}$ we can add components
$\mu^b$ and $\rho$ and then use that the two top slots of
\eqref{s2conn} have to be tracefree to deduce that
$\mu^b=-\tfrac{1}{n+1}\nabla_i\si^{ib}$ and
$$
\rho=-\tfrac{1}{n}(\nabla_i\mu^i-\Rho_{ij}\si^{ij})=
\tfrac{1}{n(n+1)}(\nabla_i\nabla_j+(n+1)\Rho_{ij})\si^{ij}.
$$
This proves the first claim and then projective invariance of $L$ can
be verified by a direct computation. It is also evident then, that
projecting to $\ker(\partial^*)/\im(\partial^*)$, one exactly obtains
the tracefree part of the top slot, which equals $D(\si)$.

In particular, we see that $\si$ is a solution of $D$ if and only if
$\nabla L(\si)$ is actually a section of the subbundle
$\im(\partial^*)\subset\ker(\partial^*)$. This suggests how the
 subclass of normal solutions is defined.

\begin{definition}\label{ndef}
  A solution $\si$ of the metricity equation $D(\si)=0$ is said to be
  {\em normal} if $\nabla L(\si)=0$.
\end{definition}

Observe that for a parallel section $s$ of $S^2\Cal T$, Proposition
\ref{nprop} implies that $s=L(\Pi(s))$, so normal solutions of the
equation \nn{metr} are in bijective correspondence with parallel
sections of the tractor bundle $S^2\Cal T$, so this gives a connection
to the holonomy of the tractor connection.

The question then is, in the case that $\si$ is non-degenerate, what
normality implies for the metric structure $g^\si$. The answer in this
case is elegant and important. Here if $n=2$ we take {\em Einstein} to mean
constant Gaussian curvature.
\begin{theorem}\label{mt}
  A non-degenerate solution $\si$ of the metricity equation \nn{metr}
  is normal if and only if the corresponding metric $g^\si$ is an
  Einstein metric.
\end{theorem}
\begin{proof}
  Assume that $\si$ is a non--degenerate solution of the metricity
  equation and $g^\si$ is the corresponding metric. From above we have
  the formula for computing $L(\si)$. By projective invariance there
  is no loss if we calculate in the scale $\tau^\si$, meaning we use
  the Levi-Civita connection $\nabla$ of $g^\si$. It follows from the
  discussion in Section \ref{intro} that this has the congenial
  consequence
$$
\nabla \si=0 .
$$
Moreover, the corresponding tensor $\Rho_{ab}$ is a non--zero multiple
of the Ricci--tensor of $g^\si$. Now we get $\mu^b=0$ and
$\rho=\tfrac{1}{n}\Rho_{ij}\si^{ij}$, so the latter is just a multiple
of the scalar curvature. Hence from \eqref{s2conn} we see that $\nabla
L(\si)$ has zero in the top slot, the trace-free part of
$\Rho_{ai}\si^{bi}$ in the middle slot and
$\tfrac{1}{n}\nabla_a(\Rho_{ij}\si^{ij})$ in the bottom slot.

By the non--degeneracy of $\si$, $\nabla L(\si)=0$ implies that
$\Rho_{ab}$ must be some multiple of $g^\si_{ab}$, which for $n\geq 3$
is precisely the Einstein condition. Conversely, assuming this and
$n\geq 3$, the scalar curvature is constant whence $\nabla L(\si)=0$.
On the other hand if $n=2$ the result follows immediately
from the described form of the bottom slot.
\end{proof}

From the formula for $L(\si)$ in the proof, we also see that for any
non--degenerate solution $\si$ of the metricity equation, the bilinear
form on each fiber of $\ce_A$ induced by the section $L(\si)$ of
$\ce^{(AB)}$ is non--degenerate if and only if the scalar curvature is
non--zero in that point. In particular, if $\si$ is a non-degenerate
normal solution, then $L(\si)$ is non--degenerate in this sense if and
only if $g^\si$ is not Ricci flat.

Our results give the perspective that the so--called Beltrami theorem,
i.e.~the characterization of projectively flat metrics, is a special
case of the link with Einstein structure described above.
\begin{corollary}\label{Beltrami}
  Let $M$ be a smooth manifold of dimension $n\geq 2$ and let $g$ be a
  pseudo--Riemannian metric on $M$ such that the projective structure
  determined by $g$ is locally projectively flat. Then $g$ has
  constant sectional curvature.
\end{corollary}
\begin{proof}
  Tautologically, $g$ determines a solution of the metricity equation
  for the projective structure determined by $g$. But it is well known
  that on locally flat structures, any solution of a first
  BGG--operator is normal (see e.g.\ \cite[(3) of Lemma 2.7]{CSS}). Hence by Theorem
  \ref{mt}, $g$ is Einstein and so has constant scalar
  curvature. Together with the vanishing of the projective Weyl
  curvature and Cotton tensor, implied by projective flatness, this
  shows that $g$ has constant sectional curvature.
\end{proof}
\noindent Note that we include the above well known result primarily
to illustrate that Theorem \ref{EM1} may be viewed as generalisation
of this. The proof of Corollary \ref{Beltrami} here may be viewed as simply a
repackaging of that given in \cite[Corollary 5.6]{EM}.

\section{Relations to other known results}

\subsection{A prolongation connection for the metricity equation}
A crucial point about the proof of Theorem \ref{mt}, that we have
given above, is that apart from general facts on first BGG operators
it only needs very simple computations. Thus the argument has the
scope to generalise easily to significantly more complicated cases. In
the special case of the metricity equation (which has been very well
studied) one may also deduce Theorem \ref{mt} directly from results of
\cite{EM}, as we now discuss. (However, the detailed analysis of the
metricity equation done in that reference would be much more
complicated to generalize.)

Since the operator $D$ is a linear partial differential operator of
finite type, one knows in general that it can be equivalently written
in first order closed form. Based on ideas on BGG sequences, this has
been done for a large class of cases (including $D$) in
\cite{BCEG}. There it is shown that solutions of these equations are
in one-to-one correspondence with parallel sections for some linear
connection on an auxiliary bundle.  In general such connections are not
unique and it is difficult to prolong invariant equations by invariant
connections. In the special case of the metricity equation, this is
precisely what has been done in \cite{EM}: the authors construct a
projectively invariant connection on $S^2\Cal T$ whose parallel
sections are in bijective correspondence with solutions of the
metricity equation.
\begin{theorem}\cite{EM} \label{EMthm}
The solutions to \nn{metr} are in one-to-one correspondence with solutions
of the following system:
\begin{equation}\label{psys}
\nabla_a\left( \begin{array}{c}
\si^{bc}\\
\mu^b\\
\rho
\end{array}
\right) + \frac{1}{n}
\left( \begin{array}{c}
0 \\
W_{ac}{}^b{}_d\si^{cd}  \\
- 2 Y_{abc}\si^{bc}
\end{array}
\right) =0.
\end{equation}
\end{theorem}
Here
\noindent \begin{equation}\label{Cotton}
Y_{abc}:= \nabla_a P_{bc}-\nabla_b P_{ac}
\end{equation} is the {\em projective Cotton tensor}, and \
$W_{ac}{}^b{}_d$ is the {\em projective Weyl tensor} (i.e. the
completely trace-free part of the full curvature). There are some sign
differences compared to \cite{EM}, as in that source the authors have
used a splitting of the tractor bundle different to that in
\cite{BEG}.

Of course, one expects to be able to recover Theorem \ref{mt}. from
\nn{psys}. Indeed this is so.  This claim amounts to showing that the
second tractor term in the display vanishes if and only if
$g^{ab}_\si$ is (the inverse of) an Einstein metric.
But it is straightforward to show that, calculating
in the scale $\nabla^g$, we have
$$
W_{ac}{}^b{}_d g^{cd} = \frac{n}{n-1}\cdot \Phi_{ac}g^{bc},
$$ 
where $\Phi_{ab}$ is the trace-free part of the Ricci tensor of
$\nabla^g$. Thus when $n\geq 3$ the term $W_{ac}{}^b{}_d g^{cd}$
certainly vanishes if and only if $g$ is Einstein.

For the cases $n\geq 3$ it remains only to verify that
$$
Y_{abc}g^{bc} =0
$$ if $g$ is Einstein. But then $g$ Einstein implies that
$P_{ab}=\lambda g_{ab}$, with $\lambda$ constant. Thus from
\nn{Cotton} it follows at once that $Y_{abc}=0$. Finally for
projective manifolds of dimension $n=2$ the tensor $W_{ac}{}^b{}_d$ is
identically zero, so we get no information at that stage. On the other
hand in this dimension (and since we calculate in the scale
$\nabla^g$) $P_{ab}$ is a multiple of the Gauss
curvature $K$ times the metric, thus $Y_{abc}=0$ if and only if $K$ is
constant.

\begin{remark}\label{KMart}
  It should also be mentioned that there are also links with the work
  \cite{KM} of Kiosak-Matveev. In that source, the authors consider
  the implications of having two distinct Levi-Civita connections in a
  projective class, at least one of which is Einstein. In this setting
  a specialisation of the system given in \nn{s2conn} (see equations
  (8), (24), and (32) in \cite{KM}) is used to prove a number of
  interesting results, including that if one of two projectively
  related Levi-Civita connections is Einstein, then so is the other,
  which was originally shown in \cite{Mikes2}. As pointed out by the
  referee, this implies that if one solution of \nn{metr} is normal
  then so are all others. In \cite{KM} it is also shown (see Theorem
  2) that in dimension 4 if two such solutions are linearly
  independent, then the structure is projectively flat.  These results
  should be visible using the tools developed here and we shall take
  that up elsewhere.
\end{remark}

\subsection{Projective holonomy}
The parallel section of $S^2\Cal T$ determined by a normal solution of
the metricity equation can be interpreted as a reduction of projective
holonomy, i.e.~the holonomy of the standard tractor connection. In the
case that this parallel section is pointwise non--degenerate (see also
the next subsection) this falls into the cases studied in the
insightful work \cite{ArmstrongP1} of S.~Armstrong. In this reference
it is shown (without discussing the related BGG equations) that, on a
set of generic points, this yields an Einstein metric whose
Levi--Civita connection lies in the projective class.

\subsection{Another first BGG equation and Klein-Einstein structures}
There is a projectively invariant differential operator
$$
K: \ce(2) \to \ce_{(abc)}(2)
$$ with leading term $\nabla_{(a}\nabla_b\nabla_{c)}$, for any
$\nabla\in p$. This is another first BGG operator, but note that it
looks very different to the (first order) metricity operator
$D$. Nevertheless normal solutions satisfying suitable non-degeneracy
conditions are again equivalent to Einstein metrics, cf.\ Theorem
\ref{mt}.  This is proved in \cite[Section 3.3]{CGHjlms} where, among
other things, the normal solutions are used to define projective
compactifications of certain Einstein manifolds; one case, that we
term a Klein-Einstein structure, is both a curved generalisation of
the Klein model of hyperbolic space and a projectively compact
analogue of a Poincar\'e-Einstein manifold. The latter is a
conformally compact negative Einstein manifold.

This interesting link is easily explained from our current
perspective. The parallel tractor $H$ arising in connection with
normal solutions to the equation $K \tau =0$ is a section of $S^2
\mathcal{T}^*$, and thus, if non-degenerate, it is equivalent to a
unique parallel section of $S^2 \mathcal{T}$, namely $H^{-1}$.  Let us
say that any solution $\si$ of the metricity equation is {\em
  algebraically generic} if the corresponding section $L(\si)\in
\Gamma(S^2 \mathcal{T})$ is everywhere non-degenerate (for normal
solutions on connected manifolds this is equivalent to non-degenerate
at one point).  From the explicit description of the splitting
operator $L$, as given in Section \ref{BGG}, we have that on the locus
where $\si$ itself is non-degenerate this non-degeneracy condition
is equivalent to the  non-vanishing of the scalar curvature of $g^\si$.
Similarly we shall say solutions of the $K \tau =0$ equation are
algebraically generic if the corresponding section of $S^2
\mathcal{T}^*$ is everywhere non-degenerate.

Now using this terminology, combined with machinery developed in
\cite{CGHjlms}, the observations above are easily rephrased as
statement in terms of the affine connections and related structures in
the projective class. Doing this we arrive at the following result,
which in part generalises Theorem \ref{mt}.
\begin{theorem} \label{mttoK}
Normal algebraically generic solutions $\si$ of the metricity equation
are naturally in one-to-one correspondence with normal algebraically
generic solutions to the equation $K \tau =0$.  This is via the
non-linear map $\si \mapsto \tau =H_{AB}X^AX^B$, where $H_{AB}$ is the
section of $S^2\mathcal{T}^*$ inverse to $(L(\si))$.  For the inverse
map: given a normal algebraically generic solution $\tau$ to $K\tau
=0$ then, on the open set where $\tau$ is not zero, the corresponding
solution of the metricity equation is $(\tau P^\tau_{ab})^{-1}$. Here
$P^\tau_{ab}$ is the Schouten tensor for the connection
$\nabla^\tau\in p$ preserving $\tau$.
 \end{theorem}
\noindent Note that $\tau$ as given above in Theorem \ref{mttoK}
agrees with $\tau^\si$, as defined in \nn{taudef}. This is easily
seen to be true, up to a constant factor, as they
have a common zero locus, and where non-vanishing are both parallel for
the Einstein Levi-Civita connection.

This Theorem means that we can at once import the results from
\cite{CGHjlms} for normal solutions to the $K$ equation and apply
these to  generic normal  solutions of the metricity equation. In
particular we have the following statement available.
\begin{corollary} If $\si$ is an algebraically generic normal solution of the
metricity equation then $\si$ is non-degenerate (and determines a metric
$g^\si$) on an open dense subset of $M$. The set where $\si$ is
degenerate is the zero locus of $\tau^\si$, and if non-empty is a
smoothly embedded hypersurface (not necessarily connected) with a
canonically induced non-degenerate conformal structure.  This
hypersurface is separating and the signature of $\si$ changes as the
hypersurface is crossed.
\end{corollary}
\noindent We have not attempted to be complete here; further results
are available by translating in an obvious way the results from
\cite[Theorem 3.2]{CGHjlms}.

\begin{remark}
  The bundle $\Lambda^{n+1}\mathcal{T}^*$ is parallelisable by the
  projective tractor connection. Let us select a ``tractor volume
  form'' $\eta$. That is $\eta\in \Gamma(\Lambda^{n+1}\mathcal{T}^*)$,
  $\eta$ is non-trivial and $\nabla \eta=0$.  Given any section $Q$ of
  $S^2\mathcal{T}$ we may take its determinant using $\eta$; let us
  denote this $\det (Q)$.  For a section of $\si^{ab}\in
  \Gamma(\ce^{(ab)}(-2))$ the condition that $\si$ is algebraically
  generic is exactly to say that $\det (L(\si))$ is nowhere zero.  On
  the other hand if $\si$ is a solution of the metricity equation
  then, on the locus where $\si$ is non-degenerate, $\det (L(\si))$
  agrees with the scalar curvature of $g^\si$ up to a non-zero
  constant; let us assume $\eta$ is chosen so that this constant 1.
  Thus, for solutions of \nn{metr}, $\det (L(\si))$ is a natural
  extension of the scalar curvature of $g^\si$ to a
  quantity that is defined everywhere on $M$, even though $g^\si$ may
  not be available globally.  In particular for normal solutions $\det
  (L(\si))$ is a constant.  On the other hand the system consisting of
  \nn{metr} plus $\det (L(\si))= \mbox{\it{constant}}$ is a weakening
  of the normality condition, and it provides a projective analogue of
  the conformal almost scalar constant equation \cite{Gal} from
  conformal geometry.
\end{remark}

\begin{remark}\label{aKM}
Interestingly there is a further link with article \cite{KM}
(cf.\ Remark \ref{KMart}). The equation $K\tau=0$ in Theorem
\ref{mttoK} may be specialised to the case that the background
connection $\nabla$ is an Einstein Levi-Civita connection. The result
is the symmetric part of the equation (34) of \cite{KM}. (The other
part of (34) is then a differential consequence of this part and the
Einstein condition.)
\end{remark}

\end{document}